\documentclass[12pt,a4paper]{amsart}
\usepackage[utf8]{inputenc}

\usepackage{amsmath}
\usepackage{amsthm}
\usepackage{amssymb}

\newcommand{\bbN}{\mathbb{N}}
\newcommand{\bbQ}{\mathbb{Q}}
\newcommand{\bbR}{\mathbb{R}}

\newcommand{\cl}{\operatorname{cl}}
\newcommand{\dist}{\operatorname{dist}}

\newcommand{\xleq}{\leqslant}

\newtheorem{thm}{Theorem}
\newtheorem{lem}[thm]{Lemma}
\newtheorem{cor}[thm]{Corollary}

\makeatletter
\@namedef{subjclassname@2020}{
  \textup{2020} Mathematics Subject Classification}
\makeatother


\title{A separate continuous function  not feebly continuous}
\author{Wojciech Bielas}
\address{Institute of Mathematics, University of Silesia in Katowice, ul. Bankowa 14, 40-007 Katowice}
\email{wojciech.bielas@us.edu.pl}

\begin{document}

\begin{abstract}
  We construct a separate continuous function $f\colon\mathbb{Q}\times\mathbb{Q}\to[0,1] $ and a dense subset $D\subseteq \mathbb{Q}\times\mathbb{Q}$ such that $f[D]$ is not dense in $f[\mathbb{Q}\times\mathbb{Q}]$, in other words,  $f$ is separate continuous and  not feebly (somewhat) continuous.
\end{abstract}

\maketitle

\section{Introduction}
The notion of a feebly continuous function was introduced and examined by Z. Frol\'{\i}k \cite{Frolik}.
Some authors say \emph{somewhat continuous function} instead of  feebly continuous function, for example \cite{Gentry}.
A function $f$ is \textit{feebly continuous}, if the preimage $f^{-1}[U]$ of any non-empty open subset $U$ has non-empty interior.
By \cite[Theorem 3]{Gentry}, any surjection $f\colon P\to Q$ is feebly continuous if and only if  the image $f[M]$ of a dense subset $M\subseteq P$ is  dense in $Q$.

If $X$ is a Baire space, $Y$ is of weight $\omega$ and $Z$ is a metric space, then any separate continuous function $f\colon X\times Y\to Z$ is feebly continuous, see \cite{Martin}.
Recall that a function $f\colon X\times Y\to Z$ is \textit{separate continuous}, if  functions $f(x,\cdot)$ and $f(\cdot,y)$ are continuous for each $(x,y)\in X\times Y$.
T. Neubrunn observed that the assumption of separate continuity can be weakened, see \cite[Theorem 2 and Theorem 3]{Neubrunn}.
Additionally, he presented counterexamples, which witness that   both conditions: $X$ being Baire and $Y$ being second countable, are necessary, see \cite[Example 3 and Example 4]{Neubrunn}.
In \cite[Example 3]{Neubrunn}, it is defined a function $f\colon (0,1)\times Y\to \{0,1\}$ such that  all functions $f(\cdot,y)$ are quasicontinuous and all functions $f(x,\cdot)$ are feebly continuous, but $f$ is not feebly continuous.
But \cite[Example 4]{Neubrunn} shows that there is a function $f\colon (\bbQ\cap(0,1))\times(1,\infty)\to \{0,1\}$ such that all sections $f(\cdot,y)$ are quasicontinuous and all sections $f(x,\cdot)$ are feebly continuous, but $f$ is not feebly continuous.

We construct a separate continuous function $f\colon\bbQ\times\bbQ\to[0,1]$ and a dense subset  $D\subseteq\bbQ\times\bbQ$ such that the image $f[\bbQ\times\bbQ]\subseteq[0,1]$ is dense and $f[D]$ is a singleton, hence $f$ is not feebly continuous.
The construction relies on the following observation.
If $(x_0,y_0),\ldots,(x_n,y_n)$ are such that $x_i\neq x_j$ and $y_i\neq y_j$ for $i<j\xleq n$, then the intersection
\[
  (\{x_n\}\times\bbQ)\cup(\bbQ\times\{y_n\})\cap \bigcup_{i< n}(\{x_i\}\times\bbQ)\cup(\bbQ\times\{y_i\})
\]
is finite.
Thus, having defined a continuous function on the set $\bigcup_{i< n}(\{x_i\}\times\bbQ)\cup(\bbQ\times\{y_i\})$, we can easily extend it to a continuous function on the set $\bigcup_{i\xleq n}(\{x_i\}\times\bbQ)\cup(\bbQ\times\{y_i\})$.
It remains to observe that there exists a dense subset $\{(x_n,y_n)\colon n\in\bbN\}\subseteq \bbQ\times\bbQ$ such that
\[  \bbQ\times\bbQ=\bigcup_{n\in\bbN}((\{x_n\}\times\bbQ)\cup(\bbQ\times\{y_n\}))
\]
with $x_i\neq x_j$ and $y_i\neq y_j$ for $i<j$.

\section{The construction}

We proceed to establish  the following lemma.

\begin{lem}\label{lem:funkcja}
  If  $(x_0,y_0),\ldots,(x_n,y_n)$ in $\bbQ \times \bbQ $ are different points and $\xi_i,\eta_i\in[0,1)$ for $0\xleq i<n$, then there exists a continuous function $f\colon (\{x_n\}\times \bbQ )\cup(\bbQ \times\{y_n\})\to[0,1]$ such that
  \begin{itemize}
  \item $f(x_n,y_n)=1$;
    \item $f(x_n,y_i)=\xi_i$ and $f(x_i,y_n)=\eta_i$,  for any $0\xleq i<n$;
        \item $\cl f[\{x_n\}\times \bbQ ]=[0,1]$.
   \end{itemize}
 \end{lem}

 \begin{proof}
   The set $A=\{(x_n,y_i)\colon i\xleq n\}\cup\{(x_i,y_n)\colon i\xleq n\}$ is finite, hence there exists a continuous function $g\colon \bbR\times\bbR\to[0,1]$ such that $g(x_n,y_n)=1$, $g(x_n,y_i)=\xi_i$ and $g(x_i,y_n)=\eta_i$, for any $i<n$. 
   Let $h\colon \bbR\times\bbR\to \bbR$ be given by the formula $$h(x,y)=\max\{0,1-\dist((x,y),A)\}$$
   for any $(x,y)\in\bbR\times \bbR$.
   The restriction  
   \[ (h\cdot g)|_{(\{x_n\}\times \bbQ )\cup(\bbQ \times\{y_n\})}=f
   \]
   is the desired function.
 \end{proof}

If  $A\subseteq X\times Y$, then the sets  $A_x=\{y\in Y\colon (x,y)\in A\}$ and $A^y=\{x\in X\colon (x,y)\in A\}$ are called \emph{sections}.
Fix a subset
$$A=\{(x_n,y_n)\colon n\in\bbN\}\subseteq \bbQ \times \bbQ $$
such that all sections $A_{x_n},A^{y_n}$ are  singletons and $\{x_n\colon n\in\bbN\}=\{y_n\colon n\in\bbN\}=\bbQ$.

\begin{thm}\label{thm:2.2}
If $A$ is defined as above, then    there exists a separate continuous function  $f\colon \bbQ \times \bbQ \to[0,1]$ such that $f[A]$ is a singleton, but the image $f[\bbQ \times\bbQ ]$ is dense.
 \end{thm}

 \begin{proof}
Let $f_0\colon (\{x_0\}\times\bbQ )\cup(\bbQ \times\{y_0\})\to[0,1]$ be any continuous function such that $f_0(x_0,y_0)=1$ and $\cl f_0[\{x_0\}\times\bbQ]=[0,1]$.
Assume that we have defined functions $f_0,\ldots,f_{n-1}$ such that if $i<k<n$, then
    \begin{enumerate}
   \item[(1)] $f_k\colon (\{x_k\}\times \bbQ )\cup (\bbQ \times\{y_k\})\to[0,1]$ is continuous,
   \item[(2)] $f_k(x_k,y_k)=1$,
     \item[(3)] $f_k(x_k,y_i)=f_i(x_k,y_i)$ and $f_k(x_i,y_k)=f_i(x_i,y_k)$.
  \end{enumerate}
  Using Lemma \ref{lem:funkcja} with parameters $\xi_i=f_i(x_n,y_i)$ and $\eta_i=f_i(x_i,y_n)$ for any $i<n$, we obtain a continuous function  $$f_n\colon (\{x_n\}\times \bbQ )\cup (\bbQ \times \{y_n\})\to [0,1]$$
  such that conditions (1)--(3) are satisfied for $i<k\xleq n$.
    Finally, let  $f\colon \bbQ \times \bbQ \to[0,1]$ be given by the formula $f(x_m,y_n)=f_m(x_m,y_n)$ for $m,n\in\bbN$.
  
 Fix $(x_m,y_n)\in \bbQ\times\bbQ$.
 For each $y_k\in\bbQ$, we have $f(x_m,y_k)=f_m(x_m,y_k)$, which implies that $f|_{\{x_m\}\times\bbQ}$ is continuous.
 By condition (3),  $f_k(x_k,y_n)=f_n(x_k,y_n)$, hence
 $f(x_k,y_n)=f_k(x_k,y_n)=f_n(x_k,y_n)$,
which implies that $f|_{\bbQ\times\{y_n\}}$ is continuous.

Functions $f$ and $f_0$ agree on the set $\{x_0\}\times \bbQ $, hence $\cl f[\{x_0\}\times\bbQ]=[0,1]$.
Clearly, we have $f[A]=\{1\}$.
 \end{proof}

\begin{cor}
There exists a separate continuous function $f\colon \bbQ \times \bbQ \to[0,1]$ which is not feebly continuous.
 \end{cor}

 \begin{proof}
   It suffices to assume that the set $A$ in Theorem \ref{thm:2.2} is also dense in $\bbQ\times\bbQ$.
 \end{proof}

\end{document}